\documentclass[12pt]{article}

\pdfoutput=1

\usepackage[english]{babel} 
\usepackage{amsmath}
\usepackage{amsfonts}
\usepackage{amssymb}
\usepackage{amsthm} 
\usepackage[utf8]{inputenc}
\usepackage{hyperref}
\usepackage{csquotes}

\hypersetup{
 colorlinks=true,
 linktoc=all, 
 linkcolor=blue, 
 filecolor=magenta,      
}
\usepackage[
 style=ext-alphabetic,
 backend=biber, 
 sorting=nyt, 
 articlein=false]{biblatex}
\renewbibmacro{in:}{}
\usepackage{csquotes}

\def\Z{\mathbb {Z}}
\def\N{\mathbb {N}}

\theoremstyle{plain}
\newtheorem{theorem}{Theorem}
\newtheorem*{theorem*}{Theorem}
\newtheorem{lemma}{Lemma}[section]
\newtheorem{hypothesis}{Hypothesis}
\newtheorem{remark}{Remark}

\author{Semchankau Aliaksei}
\title{Maximal subsets free of arithmetic progressions in arbitrary sets
\footnote{This work is supported by the Russian Science Foundation under grant 14--11--00433.}
}

\begin{document}
\maketitle

\begin{abstract}
We consider the problem of determining the maximum cardinality of a subset containing no arithmetic progressions of length $k$ in a given set of size $n$.
It is proved that it is sufficient, in a certain sense, to consider the interval $[1,\dots, n]$.
The study continues the work of Komlós, Sulyok, and Szemerédi.
\end{abstract}

\section{Introduction}

Let us consider an arbitrary set $B \subseteq \Z$ and integer $k \geqslant 3$.
We define the value $f_{k}(B)$ to be the cardinality of maximal subset of $B$, which does not contan nontrivial arithmetical progression of length $k$ (we say arithmetical progression is trivial if all of its elements are equal).
Let us consider the function 

$$
\phi_{k}(n) := \min_{|B| = n}{f_{k}(B)}.
$$

Now we introduce the function $g_{k}(n):= f_{k}([1, 2, ..., n])$.
Let $\rho_{k}(n):= g_{k}(n)/n$ be a density of maximal set free of arithmetical progressions of length $k$ in segment $[1, \dots, n]$.
We know following estimates for $\rho_{k}(n)$:
$$ 
 \frac{1}{e^{c_{k} \sqrt{\ln{n}}}} 
 \ll 
 \rho_{k}(n) 
 \ll 
 \frac{1}{(\ln{\ln{n}})^{s_{k}}} \,,
$$
where $c_{k}, s_{k}$ are positive constants, depending only on $k$. Lower bound belongs to Behrend \cite{Beh46}, and upper bound belongs to Gowers \cite{Gow01}.
Historical retrospective on special case $k = 3$ can be found in works \cite{Shk06}, \cite{Blo12}.

At first sight it seems natural to expect the equality $\phi_{k}(n) = g_{k}(n)$ to hold, although it turns out to be false already for 
$n = 5, k = 3$: $g_{3}(5) = f_{3}([1, 2, 3, 4, 5]) = 4 > 3 = f_{3}(\{1, 2, 3, 4, 7\}) = \phi_{3}(5)$. 
However, intuition still predicts that $\phi_{k}(n)$ does not differ much from $g_{k}(n)$.
In this direction it was proved by Komlós, Sulyok, and Szemerédi \cite{KSS75} that following inequality holds:
$$
\phi_{3}(n) > (1/2^{15} + o(1))g_{3}(n), n \to \infty.
$$
In O'Bryent's work \cite{OBr13} it is stated, without proof, that constant $1/2^{15}$ might be improved to $1/34$.

In here we demonstrate the following:
\begin{theorem}\label{thm:main}
For any integer $k \geqslant 3$ there exists such sequence $n_{1} < n_{2} < \dots$ of natural numbers such that for any element $n$ in it following inequality holds:
$$
\phi_{k}(n) > (1/4 + o(1))g_{k}(n).
$$
Furthermore, the sequence
 $n_{1} < n_{2} < \dots$
 is rather dense in the sense that any segment of the form $[n, ne^{(\ln{n})^{1/2 + o(1)}}]$ contains at least one element of this sequence.
 \end{theorem}
 As we see this result improves bound from \cite{KSS75} for a subsequence of $\N$. We obtain constant $1/4$ since we `compress` given set of numbers modulo a prime number twice and keep roughly half of the elements each time. 
 Our method differs from the one presented at \cite{KSS75} by fewer amount of operations (constists of 2 `compressions'), and therefore by saving more elements of the initial set.

\bigskip

For natural $n$ we denote by $[n]$ the segment $[1, \dots, n]$.

\section{Compressing Lemmas}

Let us consider some set of integers $X = \{x_{1}, x_{2}, \cdots, x_{n} \}$.
We call set $Y = \{y_{1}, y_{2}, ..., \cdots, y_{n}\}$ a \emph{compression} of set $X$, if for any triples $(i, j, k) \in [n]^{3}$ equality $x_{i} - 2x_{j} + x_{k} = 0$ implies $y_{i} - 2y_{j} + y_{k} = 0$ (notice that we do not imply any order of $x_i$ and $y_i$).
This definition is closely related to Freiman homomorphism, see \cite{TV06}. 

Now we state a hypothesis, which we prove only in special case, which however would suffice for us.

\begin{hypothesis}\label{hype:main}
For any $\epsilon > 0 $ there is such subpolynomial function $h(n) = h_{\epsilon}(n)$, such that for any integer set $X$ of size $n$ there exists such $Y \subseteq X, |Y| \leqslant \epsilon n$, for which $X \setminus Y$ might be compressed into subset of segment $[nh(n)]$.
\end{hypothesis}

We prove it for all $\epsilon \in (3/4, 1)$. 
For the sake of transparency, we break the proof into several lemmas. 
Since we are only interested in behaviour of $h(n)$ for large $n$, we would only consider a case when $n$ is large enough.

\begin{lemma}[\bf on compression into an interval of exponential length]\label{lem:first}
Any set of size $n$ might be compressed into a subset of the segment $[4n^{4}6^{n/2}]$.
\end{lemma}
\begin{proof}
Having set $X$ we want to build $Y \subset [4n^{4}6^{n/2}]$ such that $Y$ is a compression of $X$. 

We assign to $X$ a following matrix $A$. Let us enumerate all nontrivial arithmetical progressions of length $3$ in $X$: 
$$
(i_{1}, j_{1}, k_{1}), 
(i_{2}, j_{2}, k_{2}), 
\cdots, 
(i_{p}, j_{p}, k_{p}),
$$ 
where $p$ is the total amount of progressions. Clearly, for any triple $(i_{s}, j_{s}, k_{s})$ equality 
$$
x_{i_{s}} - 2x_{j_{s}} + x_{k_{s}} = 0
$$
holds.
We set $A$ to be a matrix of size $p \times n$. At $s$th row of $A$ we put $1$ at $i_{s}$th and $k_{s}$th column, and $-2$ at $j_{s}$th column. Other entries are occupied with zeros.

For example, set $X = \{ 1, 2, 3, 4, 5 \}$ would be assigned with the following matrix:
$$A = 
\begin{pmatrix}
1 & -2 & 1 & 0 & 0 \\
0 & 1 & -2 & 1 & 0 \\
0 & 0 & 1 & -2 & 1 \\
1 & 0 & -2 & 0 & 1
\end{pmatrix}
$$

It is clear from the definition of matrix $A$ that
$$
A
\begin{pmatrix}
x_{1} \\
x_{2} \\
\cdots \\
x_{n}
\end{pmatrix}
=
\begin{pmatrix}
0 \\
0 \\
\cdots \\
0
\end{pmatrix}
$$ 

Furthermore, $Y$ is a compression of $X$ if and only if
$$
A
\begin{pmatrix}
y_{1} \\
y_{2} \\
\cdots \\
y_{n}
\end{pmatrix}
=
\begin{pmatrix}
0 \\
0 \\
\cdots \\
0
\end{pmatrix}
$$ 

Let us consider an arbitrary set $X$ of size $n$ and its assigned matrix $A$: $Ax = 0$, where $x = (x_{1}, \cdots, x_{n})^{T}$. 
We would demonstrate the existion of such $y = (y_{1}, ..., y_{n})$ such that its coordinates are distinct natural numbers not exceeding $4n^{4}6^{n/2}$, satisfying $Ay = 0$.

\bigskip

Let us solve the equation $Ax = 0$. We choose maximal amount of linearly independent rows and put them to new matrix $A'$. Certainly, $A'x = 0 \Leftrightarrow Ax = 0$. 

We denote size of $A'$ by $m \times n,\ m < n$ (clearly $A$ and $A'$ are degenerate, since sum of elements in each row equals $0$). 
Let us distinguish independent (basis) variables from dependent ones. 
Clearly, there are exactly $m$ dependent variables among $x_1, x_2, \dots, x_n$. Let us swap coordinates in $x = (x_1, \dots, x_n)$ and rows in $A'$ in such a way such that first coordinates of $x$ are dependent, and last coordinates are independent. Via the Gauss elimination method we reduce the system to the following form:

$$
A''x = 
\begin{pmatrix}
1 & 0 & 0 & \cdots\\
0 & 1 & 0 & \cdots \\
0 & 0 & 1 & \cdots
\end{pmatrix}
\begin{pmatrix}
x_{1} \\
x_{2} \\
\cdots \\
x_{n} 
\end{pmatrix}
=
\begin{pmatrix}
0 \\
0 \\
\cdots \\
0
\end{pmatrix}
$$

(order of $x_{1}, x_{2}, ...$ might have changed after elimination).
By Gauss elimination property there exists such nonsingular square matrix $M$ for which $A'' = MA'$. 
Notice that this equality would still hold if we keep only first $m$ columns of $A''$ and $A'$. Therefore, if $E$ and $D'$ are corresponding square matrices, then equality $E = MD'$ ($E$ is the identity matrix) holds. 
Clearly, $M = (D')^{-1}$. 
It is known that
$$
M = (D')^{-1} = 
\begin{pmatrix}
\frac{\det(D'_{1,1})}{\det(D')} & \dots & \frac{\det(D'_{1, m})}{\det(D')}\\
\dots & \dots & \dots \\
\frac{\det(D'_{m,1})}{\det(D')} & \dots & \frac{\det(D'_{m, m})}{\det(D')}
\end{pmatrix},$$
where $D'_{i,j}$ are adjoint matrices. 
Thus, $\|\det{D'} \times M\|_{\infty}$ does not exceed the absolute value of determinant of matrix consisting of elements $1, -2, 0$, (with at most two $-1$ and at most one $2$ in each row), which we can bound by $(\sqrt{1^{2} + 1^{2} + (-2)^2})^{m} = 6^{m/2}$ by Hadamard inequality.

Since $A'$ also consists of elements $-2, 1, 0$, equality $A'' = MA'$ implies that all elements of $\det{D'}A''$ are integers with absolute values not exceeding $2m6^{m/2} \leqslant 2n 6^{n/2}$.

Now we turn to construction of desired $y = (y_{1}, ..., y_{n})$, satisfying all the conditions above. Let us consider equation $A''x = 0$ and denote its first $m$ elements by $w_{1}, ..., w_{m}$, and remaining by $z_{1}, \cdots, z_{t}$, $m + t = n$. We have:

$$A''x = 0 \Leftrightarrow 
\begin{pmatrix}
1 & 0 & 0 & \cdots\\
0 & 1 & 0 & \cdots \\
0 & 0 & 1 & \cdots
\end{pmatrix}
\begin{pmatrix}
w_{1} \\
\cdots \\
w_{m} \\
z_{1} \\
 \cdots \\
 z_{t}
\end{pmatrix}
=
\begin{pmatrix}
0 \\
0 \\
\cdots \\
0
\end{pmatrix},
$$ 

or
$$
w_{1} + a_{1, 1}z_{1} + a_{1, 2}z_{2} + \cdots + a_{1, t}z_{t} = 0,
$$
$$
\cdots
$$
$$ 
w_{m} + a_{m, 1}z_{1} + a_{m, 2}z_{2} + \cdots + a_{m, t}z_{t} = 0.\\
$$
We know that any $a_{i, j}$ becomes integer when multiplied by $\det{D'}$ not exceeding $2n 6^{n/2}$ by absolute value.
From here we obtain that for any $w_{i}$ there exists such $\alpha_{i, 1}, \alpha_{i, 2}, ..., \alpha_{i, t}$ (negative correspondent elements of $A''$, multiplied by $\det{D'}$), such that 
$$
w_{i} = \frac{\alpha_{i, 1}z_{1} + ... + \alpha_{i, t}z_{t}}{\det{D'}},
$$ 
where all of $\alpha_{i, j}$ are integer and bounded by $2n6^{n/2}$ in absolute value.

\bigskip

We now aim to find such a solution $w_{1}, \cdots, w_{m}, z_{1}, \cdots, z_{t}$, where all variables are distinct, natural and do not exceed $4n^{4}6^{n/2}$. 

Now we demonstrate that it is possible to choose from multidimensional cube $[0, K - 1]^{t}$, (where $K = n^{2}$), such $t$-tuple ($z_{1}, ..., z_{t}$), so that all elements in $y = (z_{1}, ..., z_{t}, w_{1}, ..., w_{m})$ would be distinct. Indeed, amount of possible points belonging to cube is $K^{t}$. 
Any equality of the form $z_{i} = z_{j}, z_{i} = w_{j}, w_{i} = w_{j}$ determines a hyperplane of the form $\alpha_{1}z_{1} + ... + \alpha_{t}z_{t} = 0$ - clearly, all integer points belonging to hyperplane can be projected onto the face of hypercube (and projections are integers, too). Therefore there are at most $K^{t - 1}$ integer points on any hyperplane. In total, there are at most  $C_{n}^{2}$ such hyperplanes, therefore, they contain at most$C_{n}^{2}K^{t - 1} < K^{t}$ points in total.

Having this coordinates $(z_{1}, ..., z_{t})$ we construct corresponding $w_{1}, ..., w_{m}$, multiply all elements of $y = (z_{1}, \cdots, w_{1}, \cdots)$ by $\det{D'}$ and obtain an integer-valued vector, whose maximal element does not exceed either $n^2 \times \det{D'} \leqslant n^2 \times 6^{n/2}$, (if it was one of $z_{i}$), or $n \times \max(\alpha)\times \max(z_{i}) \leqslant n \times 2n6^{n/2} \times n^{2} = 2n^{4}6^{n/2}$ (if it was one of $w_{i}$) --- and therefore we can bound maximal element as $2n^{4}6^{n/2}$. 
To get rid off negative numbers, we shift coordinates of $y$ `to right' to obtain set of naturals, with maximal element not exceeding $2\times 2n^{4}6^{n/2}$.
\end{proof}
\begin{remark}
Clearly, one cannot get rid off exponential multiplier $c^{n}$, since there is not such compression for set $\lbrace 0, 1, 2, 4, ..., 2^{n} \rbrace$ that would make maximal element less than $2^n$.
\end{remark}

\begin{lemma}[\bf on compression into subset of segment of cubic length]\label{lem:second}
If set $X$ of size $n$ belongs to segment $[1, \cdots, M]$, where $M = 4n^{4}6^{n/2}$, then there exists such subset $X' \subseteq X, |X'| \geqslant |X|/2$ which might be compressed into subset of segment $[n^{3}]$
\end{lemma}
\begin{proof}
Let us consider first prime numbers $p_{1} = 2, p_{2} = 3, p_{3} = 5, \cdots$. Let us take the minimal prime number which does not divide any difference in $X$ and denote it by $p_{k + 1}$. Then for any $p_{t}, t \leqslant k$, there are such distinct $x_{i}, x_{j}$, such that $p_{t} | (x_{i} - x_{j})$. From here we obtain 
$$
2\times 3 \times 5 \times \cdots \times p_{k} | \prod_{i\neq j}(x_{i} - x_{j}).
$$
From here we obtain the following bound (via using $p_{k} > k, |x_{i} - x_{j}| < M$):
$$
2 \times 3 \times \cdots \times k \leqslant M^{\frac{n^{2} - n}{2}}.
$$
Apply $\log$ to both parts:
$$
k\ln{k} - k \leqslant \frac{n^{2}-n}{2} \ln{M}, 
$$
from where it is easy to observe that $p_{k + 1} < 2n^3 $ for large enough $n$.

Thus, there exists such prime $p \leqslant 2n^{3}$ which does not divide any difference in $X$. 
Let us now consider a set $X' = \{x_{1} \pmod p, x_{2} \pmod p, \cdots \}$.
It has size $n$, and belongs to segment $[0, ..., p - 1]$, therefore intersects by half with one of the segments $L_{1} = [0, \cdots, p/2]$, and $L_{2} = [p/2, p - 1]$ (it is clear, that if elements form a progression in $X$, then so their images do in $X'\cap L_{i}$, provided that all of them belong to this image), and therefore we can remove at most half of the elements such that remaining set is compressible into subset of segment $[n^{3}]$.
\end{proof}

\begin{lemma}[\bf on compression into subset of segment of almost-linear length]\label{lem:third}
If set $X$ of size $n$ belongs to segment $[8n^{3}]$, then for any $\epsilon > 0$ there exists $X' \subseteq X, |X'| \geqslant (1/2 - \epsilon)|X|$ such that $[C_{\epsilon}n\ln{n}]$, where $C_{\epsilon}$ is a constant, depending only on $\epsilon$.
\end{lemma}
\begin{proof}
For $n$ sufficiently large we consider prime numbers in segment $[2n, \dots, 2cn\ln{n}]$, where $c$ is a positive constant.
By Tchebyshev theorem, when $n$ is large enough, this segment would contain at least $cn$ prime numbers. We number them as $p_{1}, p_{2}, \cdots, p_{s}$, $s > cn$. Consider triples $(i, j, t)$, where $i, j, t$ are such that $p_{t} | (x_{i} - x_{j})$. Notice that each pair $(i, j)$ of indexes participates in at most 2 triples, since $|x_{i} - x_{j}| < 8n^{3}$ and cannot be divisible by 3 or more distinct prime numbers exceeding $2n$. 
Therefore, there are at most $n^{2}$ such triples. 
By Dirichlet's box principle some $p_{t}$ corresponds to at most $n^{2}/cn = n/c$ triples. 
We remove from $X$ all $x_{i}, x_{j}$, belonging to any of this triples, and remaining set $ X_{r}$ would have size at least $(1 - \frac{2}{c})|X|.$ 

\medskip

For set $X_{r}$ it is true that difference of any two distinct elements is not divisible by any prime $p_{t} < 2cn\ln{n}$, and in the same spirit as in previous lemma we remove from $X_r$ at most half of the elements such that remaining set might be compressed into subset $X'$ of segment $[2cn\ln{n}]$. Since we can take constant $c$ arbitrary large (and, accordingly, take $n > n(c)$), we have proved the desired assertion for any $\epsilon > 0.$
 \end{proof}
 
Now we turn to a proof of the Hypothesis \ref{hype:main} in the special case $\epsilon \in (3/4, 1)$:
\begin{proof}
We assume that $\epsilon \in (3/4, 1)$. First we compress set $X$ of $n$ elements into subset of segment $[4n^{4}6^{n/2}]$ by Lemma \ref{lem:first}.
Then we throw away at most half of the elements and compress $X$ into subset of segment $[n^{3}]$ by Lemma \ref{lem:second}. 
Now we fix some $\delta > 0$ and apply Lemma \ref{lem:third} to $ X \subseteq [1, \dots, n^3] \sim [1, \dots 8(\frac{n}{2})^3]$, throw away at most $(\frac{1}{2} + \delta)\frac{n}{2}$ elements and compress remaining elements into the segment $[1, C_{\delta}\frac{n}{2}\ln{\frac{n}{2}}]$. 
In total we loose at most 
$$
\frac{n}{2} + (\frac{1}{2} + \delta)\frac{n}{2} = (\frac{3}{4} + \frac{\delta}{2})n
$$
elements, so we take $\delta$ such that inequality $\frac{3}{4} + \frac{\delta}{2} \leqslant \epsilon$ holds. 
Obviously, $\delta := 2\epsilon - \frac{3}{2} > 0$ would work.
\end{proof}

\section{Proof of Theorem \ref{thm:main}}

In what follows, we would need a following lemma:
\begin{lemma}[\bf on lower-bound for density]\label{lem:dense}
For any natural $a, b$ and natural $k\geqslant 3$ the following inequality holds:
$$
\rho_{k}(3ab) \geqslant \rho_{3}(a )\rho_{k}(b)/3.
$$
\end{lemma}
\begin{proof}
Let us bisect a segment of length $3ab$ into $a$ segments of length $3b$. Let us choose among them those, whose numbers correspond to maximal subset of segment $[a]$, free of arithmetical progressions of length $3$ (clearly, there would be exactly $g_{3}(a) = a\rho_{3}(a)$ of such segments). We bisect chosen segments into subsegments of length $b$, and only keep `middle' ones. Then we take a maximal subset free of arithmetical progressions of length $k$ of size $g_{k}(b) = b\rho_{k}(b)$ in each of these middle subsegments. Clearly, the union of all those subsets does not contain any arithmetical progression of legnth $k$, and therefore $\rho_{k}(3ab) \geqslant g_{3}(a)g_{k}(b)/3ab = \rho_{3}(a)\rho_{k}(b)/3$.
\end{proof}

Before proving Theorem \ref{thm:main}, we need following inequality:
\begin{lemma}\label{lem:phi_recursive_bound}
For large enough natural $n$, natural $k \geqslant 3$ and positive real $\alpha \in (0, 1/4)$, the following inequality holds:
$$
\phi_{k}(n) > \alpha n\rho_{k}(C_{\alpha, k}n\ln{n}).
$$
\end{lemma}
\begin{proof}
Let us consider an arbitrary set $X$ of $n$ elements. By special case of Hypothesis \ref{hype:main} with $1 - \frac{1+\alpha}{2} \rightarrow \epsilon$, one can remove at most $\epsilon n$ elements in such a way, so that remaining set might be compressed into subset $A$ of segment $[C_{\alpha}n\ln{n}]$ of size $\frac{1/4 + \alpha}{2}n$. 
Let us set $m:=C_{\alpha}n\ln{n}$. 
Now we consider $\epsilon > 0$ such that $\frac{1/4 + \alpha}{2}(1 - \epsilon) > \alpha$. 
Let us show that there exists such natural number $s$, depending only on $\alpha$, with the following property: if one considers maximal subset free of arithmetical progressions of length $k$ (which we denote by $T$) in the segment $[m + 1, m + (s + 1)m]$, then there is such a shift $A + x$ of set $A$, which has large intersection with $T$ (clearly, $|T|=(s+1)m\rho_{k}((s+1)m)$):

\begin{equation}\label{eq:shifts}
|(A+x) \cap T| \geqslant (1 - \epsilon)|A|\rho_{k}((s + 1)m). 
\end{equation}

Indeed, let us consider shifts of $A$ `to the right': $A + 1, A+2, \dots, A+sm$. Notice that any element of $T$, located left to $m + sm$, belongs to exactly $|A|$ shifts. 
Let $T_1 := T \cap [m + 1, m + sm]$ and $T_2 := T \cap [m + sm + 1, m + (s+1)m]$. 
Clearly $|T| = |T_1| + |T_2|$. Let us assume that (\ref{eq:shifts}) does not hold. 
By Dirichlet's box principle some shift of $A$ intersects $T$ by at least $|T_1| |A|/sm$ elements, and therefore one can conclude that $|T_1| |A|/sm \leqslant (1 - \epsilon)|A|\rho_{k}((s + 1)m)$, and therefore $|T_2| \geqslant |T| - |T_1| \geqslant (1 + s\epsilon)\rho_{k}((s + 1)m)$ elements of $T$. 

By Lemma \ref{lem:dense} (we assume that $s+1$ is divisible by $3$) we see $\rho_{k}(m) 
\geqslant (1 + s \epsilon)\rho_{k}((s + 1)m) \geqslant (1 + s\epsilon)\rho_{3}((s + 1)/3)\rho_{k}(m)$ (we derive leftmost inequality from the fact that set free of progressions of length $k$ cannot have density more than $\rho_{k}(m)$ on segment of length $m$). Therefore, to get a contradiction, it is enough to take $s$ to be that large so that inequality $(1 + s\epsilon)\rho_{3}((s + 1)/3) \geqslant 1$ holds. This is possible since $\rho_{3}(n) \geqslant \frac{1}{e^{c_{3}\sqrt{\ln{n}}}}$, denominator is subpolynomial, and the function $(1 + s\epsilon)\rho_{3}(\frac{s + 1}{3})$ has polynomial growth on $s$. So, we obtained required $s$ depending on $\epsilon$ and $k$, or on $\alpha$ and $k$.

So, now we have desired inequality
$\phi_{k}(n) > \frac{1/4 + \alpha}{2}(1 - \epsilon)n\rho_{k}((s + 1)m) > \alpha n \rho_{k}(H_{\alpha, k}n\ln{n})$.

\end{proof}

Now we turn to Theorem \ref{thm:main}:
\begin{proof}
Let us suppose that statement of Theorem \ref{thm:main} does not hold for some $k \geqslant 3$. 
Therefore, there exists some $\epsilon > 0$, such that for any $o(1)$ there is some segment $I = [m, me^{(\ln{m})^{1/2 + o(1)}}]$, such that for any $n \in I$ inequality $\phi_{k}(n) < (1/4 - \epsilon)g_{k}(n)$ holds. 
On the other side, by Lemma \ref{lem:phi_recursive_bound}, any $n \in I$ satisfies $(1/4 - \epsilon)g_{k}(n) \geqslant \alpha n\rho_{k}(C_{\alpha, k}n\ln{n})$, where $\alpha > (1/4 - \epsilon)$ (one can set $\alpha: = 1/4 - \epsilon/2$). 
From here we obtain that for some constant $c > 1$ ($c := \frac{\alpha}{1/4 - \epsilon}$) inequality $\rho_{k}(n) > c\rho(Cn\ln{n})$ takes place whenever $n \in I$.

Now we build the sequence $t_{1} = m, t_{2} = Ct_{1}\ln{t_{1}}, t_{3} = Ct_{2}\ln{t_{2}}, \cdots$ (we continue while $t_{i} \in I$ holds --- clearly, there are at least $(\ln{m})^{1/2 + o(1)}$ such $t_{i}$).

Therefore,
$$
\rho_{k}(t_{1}) > c\rho_{k}(t_{2}) > c^{2}\rho_{k}(t_{3}) > \cdots
$$
Now, combining lower bound for $\rho_{k}(n)$, and the fact that sequence of $t_i$ has at least $(\ln{m})^{1/2 + o(1)}$ elements, the bound $\rho_{k}(t_1) \geqslant c^{i-1}\rho_{k}(t_i)$ would yield a contradiction for the last $t_i$ in the list.
\end{proof}

\printbibliography[
heading=bibintoc,
] 

\noindent{A.S.~Semchankau

The Steklov Institute of Mathematics

119991, Russian Federation, Moscow, Ulitsa Gubkina, 8}

{\tt aliaksei.semchankau@gmail.com}

\end{document}